\documentclass[12pt,leqno]{scrartcl}

\usepackage{a4wide} 

\usepackage{isolatin1}

\usepackage{theorem}

\usepackage{amsmath}

\usepackage{amssymb}

\usepackage{amscd}

\usepackage[all]{xy}

\usepackage{times}

\usepackage{mathrsfs}

\theorembodyfont{\sl}
\newtheorem{proposition}[subsection]{Proposition}

\newtheorem{lemma}[subsection]{Lemma} 
\newtheorem{lemma2}[subsubsection]{Lemma} 
\newtheorem{theorem}[subsection]{Theorem}

\newtheorem{conjecture}[subsection]{Conjecture}

\theorembodyfont{\rmfamily}

\newtheorem{remark}[subsection]{Remark}

\newenvironment{proof}{\vspace{0.3cm}\noindent
\emph{Proof.}}{\hfill $\square$ \vspace{0.3cm}}
\newenvironment{proofarg}[1]{\vspace{0.3cm}\noindent
\emph{Proof\nobreakspace#1.}}{\hfill
$\square$ \vspace{0.3cm}}

\numberwithin{equation}{subsubsection}

\renewcommand{\mathcal}{\mathscr}

\renewcommand{\bf}{\mathbf}

\newcommand{\C}{\mathbb{C}}
\newcommand{\Nat}{\mathbb{N}}
\newcommand{\R}{\mathbb{R}}
\newcommand{\Q}{\mathbb{Q}}
\newcommand{\Z}{\mathbb{Z}}

\newcommand{\Af}{\mathbb{A}_f}

\newcommand{\Ade}{\mathbb{A}}

\renewcommand{\H}{\mathrm{H}}
\newcommand{\IH}{\mathrm{IH}}
\newcommand{\IC}{\mathrm{IC}}

\newcommand{\GL}{\bf{GL}}

\newcommand{\SL}{\bf{SL}}

\newcommand{\GU}{\bf{GU}}

\newcommand{\Ar}{\mathrm{A}}

\newcommand{\G}{\bf G}

\newcommand{\der}{\mathrm{der}}
\newcommand{\ad}{\mathrm{ad}}

\newcommand{\F}{\mathcal{F}}

\newcommand{\adj}{\mathrm{adj}}  

\newcommand{\fl}{\longrightarrow}

\newcommand{\fle}{\longmapsto}

\DeclareMathOperator{\Gal}{Gal} 
\newcommand\ggoth{\mathfrak{g}}
\newcommand{\Hecke}{\mathcal{H}}

\newcommand{\K}{\mathrm{K}}

\newcommand{\Lf}{\mathcal{L}} 
\newcommand{\M}{\mathcal{M}} 

\mathchardef\mhyphen="2D

\DeclareMathOperator{\Ob}{\mathrm{Ob}}
\newcommand\oQ{\overline{\Q}}

\newcommand\quash[1]{}

\newcommand\Sgoth{\mathfrak{S}}

\newcommand{\sous}{\setminus}
\DeclareMathOperator{\Spec}{Spec}

\newcommand{\Tr}{\mathrm{Tr}} 

\newcommand{\X}{\mathcal{X}}

\title{The standard sign conjecture on algebraic cycles: the case of Shimura varieties}
\author{Sophie Morel and Junecue Suh}

\begin{document}

\maketitle

\section{Introduction}

First we recall the standard sign conjecture, its origin, statement and significance.

Let $k$ be a field and fix a Weil cohomology theory $\H^*$ on the category of 
smooth projective varieties over $k$ with coefficients in a field $F$ of characteristic zero
(\textit{cf.} \cite{An} 3.3.1).
Denote by $M_{hom}(k)_F$ the category of homological motives 
over $k$ associated with $\H^*$, and by $M_{num}(k)_F$ the category 
of numerical motives over $k$, both with coefficients
in $F$ (\textit{cf.} \cite{An} 4.1). 
The functor $\H^*$ defines a realization functor from $M_{hom}(k)_F$ to the
category of graded $F$-vector spaces, that we will denote by $\H^*$ (\textit{cf.}
\cite{An} 4.2.5). 

The K{\"u}nneth standard conjecture (cf. \cite{An} 5.1.1) states that, 
for every smooth projective variety $X/k$, the K{\"u}nneth projectors $p^i_X$ 
onto the direct factor $\H^i(X)$ of $\H^{\ast}(X)$ are given by algebraic cycles.
In the classical theory of motives (of Grothendieck), one uses it to modify
the sign in the commutativity constraints in the $\otimes$-structure in order to
get the Tannakian category of homological motives.

For classical cohomology theories, the conjecture would be a consequence 
of the Hodge conjecture over the complex numbers
and the Tate conjecture over finitely generated fields. 
Namely, as noted by Grothendieck (see \cite{Ta} p.99), 
the projectors (and any linear combination thereof) clearly 
are morphisms of Hodge structures and commute with the Galois action,
and these cohomology classes make natural test cases
for the Hodge and Tate conjectures.

The strongest evidence for the K\"unneth conjecture is given by Katz and Messing (\cite{KM}, Theorem 2): 
It is true when $k$ is algebraic over a finite field and $\H^{\ast}$ is 
either the $\ell$-adic cohomology for a prime $\ell\neq \mathrm{char}(k)$,
or the crystalline cohomology.

For the purpose of modifying the commutativity constraints and getting the Tannakian category,
one needs somewhat less, 
and the necessary weakening is called the standard ``sign'' conjecture (terminology proposed by Jannsen;
\textit{cf.} \cite{An} 5.1.3 for the formulation and see \cite{An} 6.1.2.1 for obtaining the Tannakian category):

\begin{conjecture} For every $M\in\Ob M_{hom}(k)_F$, there exists a
decomposition $M=M^+\oplus M^-$ such that $\H^*(M^+)$ (resp.
$\H^*(M^-)$) is concentrated in even (resp. odd) degrees. 
\end{conjecture}
Equivalently, for every smooth projective $X/k$, the sum $p^{+}_X$ 
(resp. $p^{-}_X$) of the even (resp. odd) K\"unneth projectors on $\H^*(X)$ 
is given by an algebraic cycle.
Note that such a decomposition is necessarily unique (up to unique isomorphism). 

This conjecture has, in addition to the consequences in terms of
the category of homological motives and the algebraicity of (Hodge or Tate) cohomology classes,
also the following interesting consequence, due to
André and Kahn. Recall that the numerical and the homological equivalences
on algebraic cycles on projective smooth varieties are conjectured to be the same.

\begin{theorem}(\cite{An} 9.3.3.3) Let $\M$ be an additive $\otimes$-subcategory
of $M_{hom}(k)_F$ and let $\M_{num}$ be its image in $M_{num}(k)_F$.
If the sign conjecture is true for every object of $\M$, then the functor
$\M\fl\M_{num}$ admits a section compatible with $\otimes$, unique up to
$\otimes$-isomorphism.

\end{theorem}

Next we turn to the main geometric objects of this paper, Shimura varieties.
In this paper, we will take for $k$ a subfield of $\C$, and
for $\H^*$ the cohomology theory that sends a smooth projective
variety $X$ over $k$ to the Betti cohomology of $X(\C)$ with coefficients
in number fields $F$. 

Let $(\G,\X,h)$ be pure Shimura data (cf \cite{De-VS}
2.1.1 or \cite{P2} 3.1), $E\subset\C$ the reflex field 
and $\K$ a neat open compact subgroup of $\G(\Af)$. 
Denote by $S^\K$ the Shimura variety at level $\K$
associated to $(\G,\X,h)$; it is a smooth quasi-projective variety over $E$.
Assume that $E\subset k$. If $S^\K$ is projective,
denote by $M(S^\K)$ the image of $S^\K$ in
$M_{hom}(k)_{\Q}$.



For a general connected reductive group $\G$ over $\Q$, 
we say that $\G$ satisfies condition (C), if 
\begin{itemize}
\item[(i)] Arthur's conjectures (cf section \ref{L2A}) are known for $\G$,
\item[(ii)] the cohomological Arthur parameters for $\G$ satisfy a certain
condition that will be spelled out at the end of section \ref{L2A} (roughly,
that what happens at the finite places determines the parameter)
and
\item[(iii)] the classification of cohomological representations of $\G(\R)$ giv
en
by Adams and Johnson in \cite{AJ} agrees with the classification given
by Arthur's conjectures.

\end{itemize}

Given the current state of knowledge of (C) (see below), 
we will also consider a weaker condition. We say that $\G$ satisfies condition (C\'{}), 
if there exists a $\Q$-algebraic subgroup $\G'$ of $\G$
which contains the derived group $\G^{\der}$ and satisfies (C).


The goal of this paper is to prove the following theorem :

\begin{theorem}\label{thA}
Let $(\G,\X,h)$ be simple PEL Shimura data. 
Assume that $\G$ is anisotropic over $\Q$ modulo its center (so that
$S^\K$ is projective and smooth) and that it satisfies condition (C\'{}).

Then $M(S^\K)\in M_{hom}(k)_{\Q}$ satisfies the sign conjecture.

\end{theorem}

If $\G$ is not anisotropic modulo its center, then $S^\K$ is not projective,
so we can not talk about its homological motive. A possible generalization
is the motive representing the intersection cohomology of the
minimal compactification of $S^\K$. Such a motive is not known to
exist for general varieties (though we certainly expect that it does),
but in the case of minimal compactifications of Shimura varieties
it has been constructed by Wildeshaus, even in the category of Chow motives
over $E$ (cf \cite{Wil} Theorems 0.1 and 0.2).
We then have the following generalization of the previous theorem :

\begin{theorem}\label{thB}
Let $(\G,\X,h)$ be simple PEL Shimura data, and assume that $\G$ satisfies condition (C\'{}).
Denote by $IM(S^\K)$ the ``intersection motive'' of the minimal compactification of $S^\K$.

Then $IM(S^\K)$ satisfies the sign conjecture.

\end{theorem}

We actually have versions of these two theorems for motives 
with coefficients in smooth motives (whose Betti realizations
are automorphic local systems), see Theorem \ref{th-sign-int-coeff}.

We will deduce Theorems \ref{thA} and \ref{thB} from another result that we
need some more notation to state. Let $\G$ be a connected reductive group over $\Q$,
$\Ar_\G$ the maximal $\Q$-split torus in the center of $\G$, $\K_\infty$ a maximal
compact subgroup of $\G(\R)$, $\K_\infty'=\Ar_G(\R)^\circ\K_\infty$, and $\X=\G(\R)/\K'_\infty$.
We assume that $\X$ is a Hermitian symmetric domain; 
this is satisfied by the group $\G$ in any Shimura data,
and also by any subgroup thereof as in condition (C\'{}).

With the notation and under this assumption, the double coset space
\[
S^\K=\G(\Q)\sous(\X\times\G(\Ade_f)/\K)
\]
still makes sense for open compact subgroups $\K$ of $\G(\Ade_f)$, 
and is a finite disjoint union of quotients of Hermitian symmetric domains
by arithmetic subgroups of $\G(\Q)$. In particular, it is a disjoint union 
of locally symmetric Riemannian manifolds if $\K$ is neat. 
Moreover, by a theorem of Baily and Borel (cf \cite{BB} Theorem 10.11), 
$S^\K$ is a quasi-projective complex algebraic variety, smooth if $\K$ is neat.

Let $\Hecke_\K=C_c^\infty(\K\sous\G(\Af)/\K,\Q)$ be the algebra 
of functions $\K\sous\G(\Af)/\K\fl\Q$ that are locally constant and have compact support,
with multiplication given by the convolution product. 

Let $j:S^\K\fl\overline{S}^\K$ be the
embedding of $S^\K$ in its minimal compactification. Let $W$ be an
irreducible algebraic representation of $\G$ defined over a field $F$,
and denote by $\F_W$ the associated $F$-local system on $S^\K$ (cf \cite{KR}, p 113).
Let $d$ be the dimension of $S^\K$ (as an algebraic variety).
The \emph{intersection complex} of $\overline{S}^\K$
with coefficients in $\F_W$ (or $W$) is the complex
\[\IC^\K_W:=(j_{!*}(\F_W[d]))[-d].\]
The \emph{intersection cohomology} of $\overline{S}^\K$ with coefficients in
$\F_W$ (or $W$) is
\[\IH^*(S^\K,W):=\H^*(\overline{S}^K,\IC^\K_W).\]
It admits an $F$-linear action of $\Hecke_\K \otimes_{\Q} F$ (cf \cite{KR} p 122-123).

We write
\begin{equation}
\label{eqn-Hecke-dec}\tag{1}
\IH^i(S^\K,W)\otimes_F \C=\bigoplus_{\pi_f}\pi_f^{\K}\otimes\sigma^i(\pi_f),
\end{equation}
where the sum is over all irreducible admissible representations
$\pi_f$ of $\G(\Af)$, $\pi_f^\K$ is the space
of $\K$-invariant vectors in $\pi_f$ (a representation of
$\Hecke_K\otimes\C$) and the $\sigma^i(\pi_f)$ are finite-dimensional
$\C$-vector spaces. Then :

\begin{theorem}\label{thC} Assume that $\G$ satisfies (C) and let $\pi_f$ be as above. 
Then, either $\sigma^i(\pi_f)=0$ for every $i$ even, 
or $\sigma^i(\pi_f)=0$ for every $i$ odd.

\end{theorem}

\begin{remark} In general, for $\pi_f$ fixed, there can be several
degrees $i$ with $\sigma^i(\pi_f)\not=0$ (as is clear on the formula for
$\IH^i$ in section \ref{proof}).
Hence the methods of this paper cannot be used to prove the full
Künneth conjecture.

This is also clear from the fact that the Lefschetz operator on $\IH^*$
commutes with the action of the Hecke operators. Note that
the action of $\C^\times$ on the $\IH^i$ that gives the pure Hodge structure
(whose existence follows from M. Saito's theory of mixed Hodge modules)
also commutes with the action of the Hecke operators (and with the
Lefschetz operator).
See page 8 of Arthur's review paper \cite{A-L2bis} for a more
precise version of these two statements.

\end{remark}


Here is the present state of knowledge about condition (C) :
\begin{itemize}

\item[(i)] Arthur's conjectures (with substitute parameters) are known
for split symplectic and quasi-split special orthogonal groups, by the book
\cite{A-livre} of Arthur, modulo the stabilization of the twisted trace
formula and a local theorem at the archimedean place (see the end of the
introduction of \cite{A-livre}). They are also known for quasi-split
unitary groups by work of Mok (\cite{Mok}) and for their inner forms by work
of Kaletha-Minguez-Shin-White (\cite{KMSW}), 
modulo the same hypotheses. Finally, still assuming the same hypotheses,
the conjectures are known for tempered representations of
split general symplectic and quasi-split
general orthogonal groups, by work of Bin Xu
(\cite{Xu}).\footnote{Note that we only need condition (C') for theorem
\ref{thB}, so Arthur's results already allow us to get theorem \ref{thB} for
the Shimura varieties of split general symplectic groups.}

\item[(ii)] This condition, in the cases where
Arthur's conjectures are (almost) known, follows easily from strong
multiplicity one for the groups $\GL_n$.

\item[(iii)] The agreement of the classifications of Arthur and Adams-Johnson 
for cohomological representations of $\G(\R)$ is still open, though
it should be accessible.

\end{itemize}

\begin{remark} In the case of even special orthogonal groups,
Arthur's methods don't
allow to distinguish between a representation and its
conjugate under the orthogonal
group. This doesn't affect the methods of this paper and thus is not a problem
for us, see the end of
section \ref{proof}.

\end{remark}

In the final section, we discuss the possibility of using finite correspondences
to attack the standard K\"unneth (or sign) conjecture for general projective smooth varieties.

\paragraph*{Acknowledgments}

We thank Pierre Deligne for helpful discussions, especially regarding the material on finite correspondences in the final section. We also thank Colette Moeglin
for pointing out some misconceptions about Arthur's conjectures in a
previous version of this article.

\section{Reduction to Theorem \ref{thC}}

First, we review the motivic constructions of coefficient systems
and intersection motives, by Ancona and Wildeshaus.
This will allow us to state Theorem \ref{th-sign-int-coeff} with
coefficients, which, together with Theorem \ref{thC},
implies both \ref{thA} and \ref{thB}. 
We then make certain reduction steps necessary for passing from
PEL Shimura varieties to the associated connected Shimura varieties. 
Finally we prove Theorem \ref{th-sign-int-coeff}, modulo Theorem \ref{thC}.


\subsection{Review of motivic constructions}
Let $F$ be a number field and let $W$ be a finite dimensional
algebraic representation of $\G_F$.

One applies Th\'eor\`eme 4.7 and Remarque 4.8 of \cite{Anc} 
to get a Chow motive $\tilde{\mu}(W)$ over $S^{\K}$, 
whose Betti realization over $S^{\K}(\C)$ is 
the local system corresponding to $W$. 
By construction, it is a direct sum of Tate twists 
of direct summands in the motives
$$
\pi^{(r)}_{\ast} \mathbf{1}_{A^r}
$$
where $r\ge 0$ and $\pi^{(r)}: A^r \fl S^{\K}$ is the $r$th power 
of the Kuga-Sato abelian scheme.

Then one applies\footnote{Wildeshaus' construction requires a condition, which he names (+)
and is the same as (3.1.5) in \cite{P2}, on the central torus in the Shimura data. 
It is satisfied by any PEL Shimura data: See for instance 
the analysis of the maximal torus quotient of $\G$ in \S 7 of \cite{K-PSSV}.}
the main result of \cite{Wil} (Theorems 0.1 and 0.2 
and Corollary 0.3) to obtain the intermediate extension $j_{!\ast} \tilde{\mu}(W)$
on the minimal compactification $\overline{S}^{\K}$, whose Betti
realization is naturally isomorphic to the intersection complex.

Finally, by taking the direct image of $j_{!\ast} \tilde{\mu}(W)$ 
under the structure morphism $m:\overline{S}^{\K} \fl \Spec k$, 
one gets the intersection motive $IM(S^K, W)$ 
whose Betti realization is canonically isomorphic 
to the intersection cohomology $\IH^{\ast}(S^K(\C), W)$.
\footnote{Strictly speaking, Wildeshaus' construction works only for
direct factors $N$ of $\pi^{(r)}_{\ast} \mathbf{1}_{A^r}$. Given a
Tate twist $N(m)$, one takes the $m$th Tate twist of $IM(S^{\K}, N)$,
which has Betti realization $\IH^{\ast}(S^K(\C), N(m))$.}

Also constructed in \cite{Wil} (Theorem 0.5) is an endomorphism $KgK$ of 
$IM(S^{\K}, W)$, for each double coset $KgK \in \K\sous\G(\Af)/\K$, whose
Betti realization coincides with the usual action of the Hecke operator
for the coset on the intersection cohomology. The construction uses the
compatibility of the motivic middle extension and the ``change of level''
maps $[h\cdot\;]_{\ast}$, see Theorem 0.4 of \cite{Wil}.

To avoid possible confusion, we will use the notation $\widetilde{KgK}$
for the endomorphisms of $IM(S^{\K}, W)$.\footnote{Wildeshaus does not 
construct an action of the Hecke algebra on $IM(S^{\K}, W)$. 
One expects, but does not know at the moment,
that there is a canonical choice of $\widetilde{KgK}$.}

\subsection{Sign conjecture with coefficients}
Now we are ready to state a version with coefficients. 

Let $W$ be an irreducible algebraic representation of $\G$ over a number field $F$.
We denote by $\IH^{+}$ (resp. $\IH^{-}$) the direct sum of $\IH^i(S^{\K}(\C), W)$
for $i$ even (resp. odd), and denote by $p_W^{+} = p_{W, F}^{+}$ 
(resp. $p_W^{-} = p_{W, F}^{-}$) the corresponding projector 
on $\IH^{\ast}(S^{\K}(\C), W)$.

\begin{theorem}
\label{th-sign-int-coeff}
Assume that the group $\G$ in the simple PEL data satisfies condition (C\'{}), 
so that there exists a subgroup $\G'$ which contains $\G^{\der}$ and satisfies condition (C).

Then there exists an endomorphism $\tilde{p}_W^{+}$ (resp. $\tilde{p}_W^{-}$)
of the image of $IM(S^{\K}, W)$ in $M_{hom}(k)_F$ 
whose Betti realization is $p_W^{+}$ (resp. $p_W^{-}$). 
It follows that the image of $IM(S^{\K}, W)$ in $M_{hom}(k)_F$ admits a decomposition
$$
IM(S^{\K}, W)_{hom} = IM(S^{\K}, W)_{hom}^{+} \oplus IM(S^{\K}, W)_{hom}^{-}
$$
such that $\H^{\ast} (IM(S^{\K}, W)_{hom}^{+}) = \IH^{+}$
and $\H^{\ast} (IM(S^{\K}, W)_{hom}^{-}) = \IH^{-}$ are concentrated
in even and odd degrees, respectively.
\end{theorem}

\subsection{Reduction steps}
\label{sec-red-steps}

\textit{Change of base field}: 
First, we reduce to the case where $k=\C$. 
For this, note that the vector spaces of algebraic correspondences modulo 
an adequate equivalence that is coarser than the algebraic equivalence 
(in particular the homological equivalence) are invariant 
under the change of ground field from $\bar{k}$ to $\C$. 
Then use the fact that the even and odd projectors are 
invariant under the action of $\Gal(\bar{k}/k)$.

\noindent \textit{Connected components}: 
As the minimal compactification $\overline{S}^{\K}$ is normal by construction,
its connected components and irreducible components coincide. 
Thus the even projector for $S^{\K}$ is the sum of
the even projectors for the connected components, and the sign conjecture is true
for $S^{\K}$ and $W$ iff it is true for each of its connected components.

\noindent \textit{Raising the level}:
\label{ssec-level}
Finally, we may pass from a neat level subgroup $\K$ to any level subgroup
$\K'\subset\K$.
It suffices to show that $IM(S^{\K}, W)$ is a direct factor 
of $IM(S^{K'},W)$ in $M_{\hom}(\C)_F$.
Take a connected component $\overline{X}$ of $\overline{S}^{\K}$, 
with $X = \overline{X} \cap S^{\K}$. The change of level map 
$f:=[1\cdot]: S^{\K'} \fl S^{\K}$ is a finite \'etale surjection that extends to
a finite surjection $\overline{f}: \overline{S}^{\K'} \fl \overline{S}^{\K}$. 
Let $\overline{Y}$ be the inverse image $\overline{f}^{-1}(\overline{X})$ and 
$Y := \overline{Y} \cap S^{\K'} = f^{-1} (X)$.

Over $X$ we have the adjunction map for direct image and the trace map:
$$
\adj_f: \F_W \fl f_{\ast} \F_W \,\, \mbox{ and } \,\,
\Tr_f: f_{\ast} \F_W \fl \F_W.
$$
Since $\overline{f}$ is finite, these maps extend to
$$
\overline{\adj_f} : j_{!*} \F_W \fl \overline{f}_{\ast} (j_{!*} \F_W) \,\, \mbox{ and } \,\,
\overline{\Tr_f}: \overline{f}_{\ast} ( j_{!*}\F_W )\fl j_{!*}\F_W
$$
(here $j_{!*} \F_W$ means $j_{!*}(\F_W [\dim X])[-\dim X]$). In Betti cohomology, these maps give rise to
$$
\xymatrix{
\IH^i (S^{\K}, W) \ar[r]^-{\adj} & \IH^i (S^{\K'}, W) \ar[r]^-{\Tr} & \IH^i (S^{\K}, W)
}
$$
\begin{lemma2}
The composite map is equal to multiplication by $\deg(f)$.
\end{lemma2}
\begin{proof}
We may replace $\overline{Y}$ (hence $Y$) with any connected component, denoting 
the restriction of $\overline{f}$ (also $f$) by the same letter. It suffices to show that
$$
\overline{\Tr_f} \circ \overline{\adj_f} = \deg(f) \;\; \mbox{ on } j_{!*}\F_W.
$$
By construction $\F = \F_W$ is a semisimple local system,
and we may replace it with a direct summand and assume it is irreducible.
Then $j_{!*} \F$ is a simple perverse sheaf
(\textit{cf.} Th\'eor\`eme 4.3.1(ii) of \cite{BBD}),
and it suffices to show the equality over the dense open subset $X$.
This last follows from Th\'eor\`eme 2.9 (Var 4) (I), expos\'e XVIII, SGA4.
\end{proof}

As we have recalled, Wildeshaus constructs $[1\cdot]_{\ast}$ between intersection motives;
see also the construction leading up to Corollary 8.8 in \cite{Wil}.
Thus $IM(S^{\K}, \F_W)$ is a direct factor of $IM(S^{\K'}, \F_W)$,
modulo homological equivalence.

\subsection{Proof of Theorem \ref{th-sign-int-coeff} modulo Theorem \ref{thC}}

By the previous reduction steps, we may pass to the connected Shimura varieties,
see 2.1.2, 2.1.7 and 2.1.8 in \cite{De-VS}: The projective system of connected
locally symmetric varieties depend only on the triple $(\G^{\ad}, \G^{\der}, \X^{+})$.
The motivic constructions of the coefficient systems and the intersection cohomology
can be therefore transferred to the connected locally symmetric varieties attached to
the subgroup $\G'$, for all small enough level subgroups.

\begin{lemma}
\label{lem-hecke-elts}
If Theorem \ref{thC} is true for $W$, then there exist elements
$$
h^{\pm}_{W, F} \in \Hecke_{\K} \otimes_{\Q} F
$$
which act as $p^{\pm}_W$ on $\IH^{\ast}(S^{\K}(\C), W)$.
\end{lemma}

\begin{proofarg}{of Lemma}
First we prove the statement over the field of coefficients $F'=\C$.
Let $\Sigma$ be the finite set consisting of the irreducible
admissible representations $\pi_f$ that have nonzero contribution 
to the right hand side of (\ref{eqn-Hecke-dec}) for some $i$.
Then as representations of the Hecke algebra $\Hecke_K\otimes_{\Q} \C$,
$( \pi_f^K )_{\pi_f \in \Sigma}$ are irreducible and pairwise inequivalent.
By Jacobson's density theorem, 
for each $\pi_f \in \Sigma$
there exists an element $h_{\pi_f} \in \Hecke_K \otimes_{\Q} \C$
that acts as $1$ on $\pi_f^K$ and as $0$ on $\pi'^K_f$ 
for every other $\pi'_f$ in $\Sigma$.

By Theorem \ref{thC}, $\Sigma$ is the disjoint union of
two subsets $\Sigma^{\pm}$, consisting of those $\pi_f \in \Sigma$ 
that have contribution in even or odd degrees, respectively.
Therefore
$$
h^{\pm}_{W, \C} = \sum_{\pi_f \in \Sigma^{\pm}} h_{\pi_f} \in \Hecke_{\K} \otimes_{\Q} \C
$$
acts as $p^{\pm}_{W, \C} = p^{\pm}_{W, F} \otimes_F 1_{\C}$.

To conclude the proof, use the fact: If $f: H \fl E$ is
an $F$-linear map of $F$-vector spaces and $F'$ is an extension field of $F$,
then an element $p\in E$ lies in the image of $f$ iff $p\otimes 1_{F'}$
is in the image of $f\otimes_F 1_{F'}$.
\end{proofarg}

Now Theorem \ref{th-sign-int-coeff} follows easily from the lemma: Writing
$$
h^{\pm}_{W, F} = \sum_{g \in K \sous \G(\Af)/\K} c^{\pm}_g \; [1_{KgK}], \;\; c^{\pm}_g \in F
$$
the endomorphism of $IM(S^{\K}, W)$ in $M_{rat}(\C)_F$ and also its image in $M_{hom}(\C)_F$
$$
\tilde{p}^{\pm}_W := \sum c^{\pm}_g \; \widetilde{KgK}
$$
has Betti realization $p^{\pm}_W$. 

In the case $W$ is the trivial representation defined over $\Q$,
we get Theorems \ref{thA} and \ref{thB}.

\begin{remark}
We have focused on PEL Shimura varieties, in order to apply the known constructions.
However, the deduction via Lemma \ref{lem-hecke-elts} of the sign conjecture 
from Theorem \ref{thC} is valid for more general varieties.

First, for the trivial coefficient system, the motivic construction 
of coefficient systems is unnecessary, and we do not need to restrict to PEL types.

 Then for compact Shimura varieties 
(that is, in case the group is anisotropic over $\Q$ modulo center),
we do not need Wildeshaus' construction of intersection motives,
and the Shimura data do not need to satisfy his condition (+)
on its central torus.

 Over the complex numbers (or even over $\oQ$, see \cite{Fal}),
the sign conjecture can be verified for the locally symmetric varieties 
considered in Theorem \ref{thC}. Through the work of Shimura, Deligne, 
Milne, and others we have a complete theory of canonical models 
of Shimura varieties over reflex fields, 
and the sign conjecture holds for these models.

 Finally, if we know the sign conjecture for the varieties attached to
a $\Q$-anisotropic semisimple group $\G$, we also know it for the varieties 
attached to any isogenous quotient group of $\G$. 
For any variety of the latter kind admits a finite \'etale covering
from a variety of the former kind, and we can apply 
an argument similar to the one in \ref{ssec-level}.

Of course, Theorem \ref{thC} is essential in all these generalizations.

\end{remark}

\section{Arthur's conjectures}
\label{L2A}

We follow the presentation of Kottwitz in section 8 of \cite{K-SVLR}.
As before, $\G$ is a connected
reductive group over $\Q$.

Let $\xi:\Ar_\G(\R)^\circ\fl\C^\times$ be a character of $\Ar_\G(\R)^\circ$.
Let $L^2_\G$ be the space of functions
$f:\G(\Q)\sous\G(\Ade)\fl\C$ such that :
\begin{itemize}
\item $f(zg)=\xi(z)f(g)$ $\forall z\in\Ar_\G(\R)^\circ,g\in\G(\Ade)$;
\item $f$ is square-integrable modulo $\Ar_\G(\R)^\circ$.

\end{itemize}
(Cf. the beginning of section 2 of \cite{A-L2}.)

Then $\G(\Ade)$ acts on $L^2_\G$ by right multiplication on the argument of
the function. We say that an irreducible representation $\pi$ of $\G(\Ade)$ is
\emph{discrete automorphic} if it appears as a direct summand in the
representation $L^2_\G$. In that case, we write $m(\pi)$ for the
multiplicity of $\pi$ in $L^2_\G$; it is known to be finite.
We denote by $\Pi_{disc}(\G)$ the set of equivalence classes of
discrete automorphic representations of $\G(\Ade)$ and by $L^2_{\G,disc}$
the discrete part of $L^2_\G$ (ie the completed direct sum of the
isotypical components of the $\pi\in\Pi_{disc}(\G)$).

Arthur conjectured that 
\[L^2_{\G,disc}\simeq\bigoplus_\psi\bigoplus_{\Pi_\psi}m(\psi,\pi)\pi,\]
where the $\psi$ are equivalence classes of global Arthur parameters,
the $\Pi_\psi$ are sets of (isomorphism classes of) smooth admissible
representations of $\G(\Ade)$
called Arthur packets
and $m(\psi,\pi)$ are nonnegative integers that we will define later.
Note that we are not saying that the representations $\pi$ are irreducible.
(They are not in general.)

The traditional statement of Arthur's conjectures involves the conjectural
Langlands group $\Lf_\Q$ of $\Q$, and Arthur parameters are morphisms
$\Lf_\Q\times\SL_2(\C)\fl{}^L\G$, where ${}^L\G=\widehat{\G}\rtimes W_\Q$ is
the Langlands dual group of $\G$. In some cases, it is possible to use instead
substitute parameters defined in terms of cuspidal automorphic representations
of general linear groups. This is the point of view that is taken in the
proofs of Arthur's conjectures by Arthur for symplectic and orthogonal groups
(cf \cite{A-livre}) and by Mok in the case of quasi-split unitary groups
(cf \cite{Mok}). In any case, a global Arthur parameter $\psi$ gives rise to :
\begin{itemize}
\item a character $\xi_\psi:\Ar_\G(\R)^\circ\fl\C^\times$;
\item a reductive subgroup $S_\psi$ of $\widehat{\G}$ such that
$S_\psi^\circ\subset Z(\widehat{\G})^\Gamma\subset S_\psi$, where
$\Gamma=\Gal(\overline{\Q}/\Q)$;
\item a character $\varepsilon_\psi$ of the finite group
$\Sgoth_\psi:=S_\psi/Z(\widehat{\G})^\Gamma
S_\psi^\circ$ with values in $\{\pm 1\}$.

\end{itemize}
In the sum above, we only take the parameters $\psi$ such that $\xi_\psi=\xi$.

Part of Arthur's conjectures is that there should be a map $\pi^0\fle\langle .,
\pi^0\rangle$ from the set of isomorphism classes of irreducible constituents
of elements of
$\Pi_\psi$ to $\widehat{\Sgoth}_\psi$, such that $\langle .,\pi^0\rangle=
\langle .,\pi^1\rangle$ if $\pi^0$ and $\pi^1$ are two irreducible constituents
of the same $\pi\in\Pi_\psi$, and that the
multiplicity $m(\psi,\pi)$ is given by the following formula :
\[m(\psi,\pi)=m(\psi,\pi^0):=
|\Sgoth_\psi|^{-1}\sum_{x\in\Sgoth_\psi}\varepsilon_\psi(x)
\langle x,\pi^0\rangle,\]
if $\pi^0$ is an irreducible constituent of $\pi$.

We can now state part (ii) of condition (C). It says that, for every
irreducible admissible representation $\pi_f$ of $\G(\Af)$, there is
at most one Arthur parameter $\psi$ such that 
$\pi_f$ is the finite part of an irreducible constituent of an element of
$\Pi_\psi$.

There are also local versions of Arthur's conjectures involving local
Arthur parameters and local Arthur packets. We will not give details here
(see for example chapter I of Arthur's book \cite{A-livre}).

\section{Proof of Theorem \ref{thC}}
\label{proof}

We use the notation of Theorem \ref{thC} and of section \ref{L2A}, and
we take for $\xi:\Ar_\G(\R)^\circ\fl\C^\times$ the inverse of the
character by which $\Ar_\G(\R)^\circ$ acts on $W(\R)$.

If $\pi$ is
an irreducible representation of $\G(\Ade)$, we can write $\pi=\pi_f\otimes
\pi_\infty$, where $\pi_f$ (resp. $\pi_\infty$) is an irreducible representation
of $\G(\Ade_f)$ (resp. $\G(\R)$).

Let $\ggoth$ be the complexified
Lie algebra of $\G(\R)$. If $\pi_\infty$ is an
irreducible representation of $\G(\R)$, we write $\H^*(\ggoth,\K'_\infty;
\pi_\infty\otimes W)$ for the $(\ggoth,\K'_\infty)$-cohomology of the space
of $\K'_\infty$-finite vectors in $\pi_\infty\otimes W$ (cf chapter I of
\cite{BW}).

It follows from Zucker's conjecture (a theorem of Looijenga (\cite{Lo}),
Saper-Stern (\cite{SS}) and Looijenga-Rapoport(\cite{LR})) and from
Matsushima's formula (proved by Matsushima for $S^\K$ compact and
by Borel and Casselman in the general case, cf Theorem 4.5 of
\cite{BC}) that there is a $\Hecke_\K\otimes\C$-equivariant isomorphism,
for every $k\in\Z$,
\[\IH^k(S^\K,W)\simeq\bigoplus_{\pi\in\Pi_{disc}(\G)}\pi_f^{\K}\otimes
\H^k(\ggoth,\K'_\infty;\pi_\infty\otimes W)^{m(\pi)}\]
(see also (2.2) of Arthur's article \cite{A-L2}).

If $\pi_f$ is an irreducible representation of $\G(\Ade_f)$, let $\Pi_\infty
(\pi_f)$ be the set of equivalence classes of irreducible representations
$\pi_\infty$ of $\G(\R)$ such that $\pi:=\pi_f\otimes\pi_\infty\in\Pi_{disc}
(\G)$. Then, for every irreducible admissible representation
$\pi_f$ of $\G(\Ade_f)$ and every $k\in\Z$,
\[\dim\sigma^k(\pi_f)=\sum_{\pi_\infty\in\Pi_\infty(\pi_f)}m(\pi_f\otimes
\pi_\infty)\dim\H^k(\ggoth,\K'_\infty;\pi_\infty\otimes W).\]

Vogan and Zuckerman have classified all the admissible representations
$\pi_\infty$ of $\G(\R)$ such that $\H^*(\ggoth,\K'_\infty;\pi_\infty\otimes W)
\not=0$ in \cite{VZ}, and Adams and
Johnson have constructed local Arthur packets
for these representations in \cite{AJ}. (It is part of our assumptions
that their construction is compatible with the local and global
Arthur conjectures of
section \ref{L2A}.) We will follow Kottwitz's exposition of their results
in section 9 of \cite{K-SVLR}.

Let $\theta$ be the Cartan involution of $\G(\R)$ that is the identity on
$\K_\infty$. For every real reductive group $H$, let $q(H)=\frac{1}{2}\dim(H/
K_H)$, where $K_H$ is a maximal compact-modulo-center subgroup of $H$.

Fix $\pi_f$ such that $\Pi_\infty(\pi_f)\not=\varnothing$. By part (ii)
of condition (C), $\pi_f$
determines a global Arthur parameter $\psi$, and we write
$\psi_\infty$ for the local Arthur parameter of $\G_\R$ defined by $\psi$.
The set $\Pi_\infty(\pi_f)$ is a subset of the local
Arthur packet associated to $\psi_\infty$.
If $\pi_\infty\in\Pi_\infty(\pi_f)$ and $\pi=\pi_f\otimes\pi_\infty$, then
the character $\langle .,\pi\rangle$ of $\Sgoth_\psi$ factors as
$\langle .,\pi_f\rangle\langle .,\pi_\infty\rangle$, where both factors
are characters of $\Sgoth_\psi$, and the first (resp. second) factor depends
only on $\pi_f$ (resp. $\pi_\infty$). By the multiplicity formula in
section \ref{L2A}, the fact that $m(\pi)=m(\psi,\pi)\not=0$ means that
the character $\langle .,\pi_\infty\rangle$ of $\Sgoth_\psi$ is uniquely
determined by $\pi_f$.
\footnote{$\langle ,.\pi_\infty\rangle$ is actually a character of the bigger
group $\Sgoth_{\psi_\infty}$, but its values on $\Sgoth_{\psi_\infty}$
are not determined by $\pi_f$, otherwise $\Pi_{\infty}(\pi_f)$ would be a
singleton, and this is not the case in general (cf case 3 on page 90 of
Rogawski's paper \cite{Ro} for a counterexample if $\G=\GU(2,1)$).}

Let $\pi_\infty\in\Pi_\infty(\pi_f)$. Then there is a relevant
pair $(L,Q)$ such that $\pi_\infty$ comes by cohomological induction from
a 1-dimensional representation of $L$, cf pages 194-195 of \cite{K-SVLR}.
Here $Q$ is a parabolic subgroup of $\G_\C$ and $L$ is a Levi component of
$Q$ that is defined over $\R$. By proposition 6.19 of \cite{VZ},
$\pi_\infty\otimes W$ can only have $(\ggoth,\K'_\infty)$-cohomology in degrees
belonging to $R+2\Nat$,
with $R=\dim_\C(\mathfrak{u}\cap\mathfrak{p})$, where $\mathfrak{u}$
is the Lie algebra of the unipotent radical of $Q$ and $\mathfrak{p}$ is
the $-1$-eigenspace for $\theta$ acting on $\ggoth$. Let
$\mathfrak{l}$ be the (complex) Lie algebra of $L$. As $\mathfrak{l}$
and $\mathfrak{u}$ are invariant under $\theta$ (by construction of $L$ and
$Q$), we see easily that
\[\dim_\C(\mathfrak{p})=2R+\dim_\C(\mathfrak{l}\cap\mathfrak{p}),\]
hence $R=q(\G_\R)-q(L)$. So the parity of $R$ is determined by the parity of
$q(L)$.

Now lemma 9.1 of \cite{K-SVLR} says that
\[(-1)^{q(L)}=\langle \lambda_{\pi_\infty},s_\psi\rangle,\]
where $s_\psi\in\Sgoth_\psi$ is determined by the global parameter $\psi$
(if we see global parameters as morphisms $\psi:\Lf_\Q\times\SL_2(\C)\fl
{}^L\G$, then $s_\psi$ is the image by $\psi$ of the nontrivial central
element of $\SL_2(\C)$) and $\lambda_{\pi_\infty}$ is the character of
$\Sgoth_{\psi_\infty}\supset\Sgoth_\psi$ defined on page 195 of
\cite{K-SVLR}. But lemma 9.2 of \cite{K-SVLR} implies that the product
$\lambda_{\pi_\infty}\langle .,\pi_\infty\rangle$ is independent of
$\pi_\infty$ in the Arthur packet of $\psi_\infty$,
so the restriction of $\lambda_{\pi_\infty}$ to $\Sgoth_\psi$ depends only on
$\pi_f$. This implies that the parity of $R$ depends only on $\pi_f$,
which gives Theorem \ref{thC}.

We have to be a bit careful if $\G$ is a quasi-split even special
orthogonal group, because in that case Arthur proved his conjectures only
up to conjugacy by the quasi-split even orthogonal group $\G'\supset\G$.
But, if $\pi_\infty$ is a representation of $\G(\R)$ with nonzero
$(\ggoth,\K'_\infty)$-cohomology, then the integer $R$ associated to
$\pi_\infty$ as above does not change if we replace $\pi_\infty$ by a
$\G'(\R)$-conjugate (because the relevant pair $(L,Q)$ is just
replaced by a $\G'(\R)$-conjugate). So the proof above still applies.

\section{K\"unneth conjecture and finite correspondences}

From the proofs of the theorem of Katz and Messing and that of ours,
one may wonder if the K\"unneth conjecture or the sign conjecture
can be proved for more general projective smooth varieties,
only using finite correspondences. More precisely, consider the $\Q$-subspace
$$
Z_{\mathrm{fin}, \H^{\ast}}^d \subseteq \H^{2d}(X\times_k X)(d)
$$
spanned by the cohomology classes of all the cycles of codimension $d$ on $X\times_k X$,
that are \textit{finite} in both projections to $X$ (where $d=\dim X$).

\begin{conjecture}
\label{conj-fin-corr}
For every projective smooth variety $X/k$ and every $i\in \Z$, the K\"unneth projector
$\pi^i_X$ (resp. the projector $\pi^{+}_X$) belongs to $Z^d_{\mathrm{fin}, \H^{\ast}}$.
\end{conjecture}
This is a priori stronger than the K\"unneth (resp. the sign) conjecture.
It turns out that the apparent strength is only illusory, if either (a) $k$ is algebraically closed
or (b) $k$ is perfect and $H^{\ast}$ is a classical Weil cohomology theory. 
The case (a) is a consequence of the following proposition.

\begin{proposition}
\label{moving-lemma}
Suppose that $k$ is an algebraically closed field.  Then the abelian group
$$
Z_{\mathrm{fin}, \sim_{\mathrm{rat}}}^d \subseteq Z_{\sim_{\mathrm{rat}}}^d
$$
generated by the cycles mapping finitely to $X$ in both projections, 
in the group of codimension $d$ cycles on $X$ modulo rational equivalence,
is in fact equal to the whole $Z^d_{\sim_{\mathrm{rat}}}$.
\end{proposition}
\begin{proof}
It is enough to prove that any irreducible closed subscheme of codimension $d$ on $X\times_k X$
is rationally equivalent to a cycle that is finite in both projections.
Because a proper quasi-finite map is finite, it is the matter of finding a cycle in the rational equivalence class,
that meets all the closed fibres over $k$-rational points in both projections properly, that is, in dimension at most zero.
This follows from the generalized moving lemma \cite{FL} of Friedlander and Lawson:
In any fixed projective embedding, all the fibres of the first (resp. second) projection
have the same degree, as they are all algebraically equivalent.
\end{proof}

Now, in the case (b), let $k$ be a perfect field, and suppose that 
$Z$ is an algebraic cycle of codimension $d$ on $X\times_k X$.
If $\H^{\ast}$ is a classical Weil cohomology theory, then we have a corresponding
cohomology theory $\H^{\ast}_{/k'}$ for every algebraic extension $k'$ of $k$,
compatible with the cycle class maps in an obvious sense.

Let $\overline{k}$ be an algebraic closure of $k$. By Proposition \ref{moving-lemma}, $Z\otimes_k \overline{k}$ is rationally --- 
hence homologically --- equivalent to a cycle $Z'$
which is finite over $X \otimes_k \overline{k}$ in both projections.
Let $k'$ be a finite Galois extension of $k$ over which $Z'$ is defined.
Taking the ``average'' of the $\Gal(k'/k)$-translates of $Z'$
(which requires $\Q$-coefficients), 
one gets a cycle $Z'_0$ on $X$, defined over $k$, that is finite in both projections
and has the same cohomology class as $Z$.

This means that, in the two cases, if the K\"unneth conjecture is true for $X/k$, then each $\pi^i_X$
is in fact a linear combination of the cohomology classes of finite correspondences over $X$.
Finding enough such finite correspondences for general $X/k$ 
(which can be turned into the problem of finding certain finite extensions of the function field $k(X)$)
seems to be an interesting open problem.

\bibliographystyle{plain}
\bibliography{sign}

\end{document}